\documentclass{amsart}
\usepackage[utf8]{inputenc}
\usepackage{amsfonts}
\usepackage{hyperref}
\usepackage{enumitem}
\usepackage{amsmath}
\usepackage{amsthm}
\usepackage[capitalize,nameinlink]{cleveref}
\usepackage{comment}
\usepackage{mathrsfs}
\usepackage{mathabx}
\usepackage{tikz}
\usepackage{todonotes}
\usetikzlibrary{positioning}
\usepackage{caption}
\newtheorem{theorem}{Theorem}[section]
\usepackage[a4paper, margin=1in]{geometry}
\newtheorem{corollary}[theorem]{Corollary}
\newtheorem{lemma}[theorem]{Lemma}

\newtheorem{proposition}[theorem]{Proposition}
\newtheorem{remark}[theorem]{Remark}
\newtheorem{definition}[theorem]{Definition}
\newtheorem{example}[theorem]{Example}

\makeatletter
\newcommand{\ostar}{\mathbin{\mathpalette\make@circled\star}}
\newcommand{\make@circled}[2]{%
  \ooalign{$\m@th#1\smallbigcirc{#1}$\cr\hidewidth$\m@th#1#2$\hidewidth\cr}%
}
\newcommand{\smallbigcirc}[1]{%
  \vcenter{\hbox{\scalebox{0.77778}{$\m@th#1\bigcirc$}}}%
}
\makeatother

\title{Metaplectic Quantum Time--Frequency Analysis, Operator Reconstruction and Identification}
\author{
 Henry McNulty 
  }
\begin{document}

\begin{abstract}
    The problem of identifying and reconstructing operators from a diagonal of the Gabor matrix is considered. The framework of Quantum Time--Frequency Analysis is used, wherein this problem is equivalent to the discretisation of the diagonal of the polarised Cohen's class of the operator. Metaplectic geometry allows the generalisation of conditions on appropriate operators, giving sets of operators which can be reconstructed and identified on the diagonal of the discretised polarised Cohen's class of the operator.
\end{abstract}

\maketitle

\section{Introduction}
The problem of discretising and reconstructing operators from discrete measurements is longstanding within the fields of wireless communication and Time--Frequency Analysis (TFA). A central object of study in this problem is the Gabor, or channel matrix \cite{Gr06} \cite{CoGrNiRo13} \cite{CoGrNi12}
\begin{align}\label{gabormatrix}
     \{\langle S\pi(\lambda)g,\pi(\mu)g\rangle\}_{\lambda,\mu\in\Lambda},
\end{align}
for some operator $S$, suitable window function $g$ and lattice $\Lambda\in\mathbb{R}^{2d}$. A natural question is to what extent we can reconstruct or identify an operator based on the Gabor matrix, or certain parts of it, especially the main diagonal. Closely related is the problem of distinguishing operators based on their action on a single function, namely, given a space of operators $\mathcal{H}\subset \mathcal{L}(X,Y)$, does there exist an $f\in X$ and positive constants $A,B$ such that
\begin{align}\label{operatoridentification}
    A\|S\|_\mathcal{H} \leq \|Sf\|_{Y} \leq B\|S\|_\mathcal{H}
\end{align}
for every $S\in\mathcal{H}$. 

In the rank--one case, the question of uniquely identifying an operator from the diagonal of its Gabor matrix is the problem of phase retrieval of the Short--Time Fourier Transform (STFT) on lattices, since for $S=f\otimes f$;
\begin{align*}
    \langle S\pi(\lambda)g,\pi(\lambda)g\rangle = |V_g f(\lambda)|^2.
\end{align*}
It is well known that reconstruction from the diagonal of the Gabor matrix is impossible in general for the class of Hilbert-Schmidt operators. In fact, the map $S\mapsto \{\langle S\pi(\lambda)g,\pi(\lambda)g\rangle\}_{\lambda\in\Lambda}$ cannot ever be injective for any lattice $\Lambda$ \cite{Sk20}. Even restricting the set of operators to the rank--one self--adjoint operators where the problem reduces to phaseless sampling of the STFT, the mapping $f \mapsto \{|V_f g(\lambda) |\}_{\lambda\in\Lambda}$ is never injective \cite{GrLi22}. Instead, we must restrict the set of operators under consideration.

The most well-known results regarding reconstruction from the Gabor matrix diagonal in this direction involve the so--called underspread operators, which have a compactly supported Fourier-Wigner transform. In this case, complete reconstruction is possible from the main diagonal, or indeed any side diagonal \cite{GrPa13}:
\begin{theorem}[Proposition 2, \cite{GrPa13}, Theorem 7.4, \cite{Sk20}]\label{intro:diagreconstr}
    Given a lattice $\Lambda \subset \mathbb{R}^{2d}$ with adjoint lattice $\Lambda^\circ = A\mathbb{Z}^{2d}$, let $Q$ denote the fundamental domain of $\Lambda^\circ$ given by $Q=A[-\tfrac{1}{2},\tfrac{1}{2}]^{2d}$. Let $h\in C^{\infty}_c(\mathbb{R}^{2d})$ such that $h|_{(1-\epsilon)Q}\equiv 1$ and $\mathrm{supp} (h) \subset Q$, and $g\in\mathscr{S}(\mathbb{R}^d)$ such that $V_g g$ is non-zero in $\mathrm{supp}(h)$. Then given $T\in\mathfrak{S}^\prime$ with $\mathcal{F}_W(T)\subset (1-\varepsilon)Q$ where $0<\varepsilon<1/2$;
    \begin{align*}
        T = \sum_{\lambda\in\Lambda} \langle T\pi(\lambda)g, \pi(\lambda)g\rangle \alpha_\lambda (R)
    \end{align*}
    where $\mathcal{F}_W(R)=\frac{h}{|\Lambda|\cdot A(g,g)}$.
\end{theorem}
By restricting the class of operators to those with Weyl symbols in a sufficiently nice space, the Feichtinger algebra, this class also satisfies the condition of \cref{operatoridentification}, with respect to the test distribution of translated Dirac delta distributions, or the Shah distribution \cite{PfWa06} \cite{PfWa09} \cite{KoPf05}.

In this paper we consider operator reconstruction and identification through the lens of Quantum Time--Frequency Analysis and metaplectic analysis of operators. Introduced by Werner in 1984 \cite{We84}, Quantum Harmonic Analysis (QHA) introduces operations from harmonic analysis to certain classes of operators. By identifying an operator with its Weyl symbol, translations, convolutions and an appropriate Fourier transform, the Fourier-Wigner transform, are defined for operators. Upon taking rank--one operators in these operations, one retrieves many familiar objects from time--frequency analysis \cite{LuSk18} \cite{LuSk19}, as one might expect, since both time--frequency analysis and Weyl quantisation are based on the representation of the Heisenberg group. By extending the setting of QHA to include modulations of the Weyl symbol of an operator, we can use a similar approach to time--frequency analysis, now on the operator level, leading to the notion of Quantum Time--Frequency Analysis (QTFA). The analogue to the time--frequency shifts on operators gives rise to the polarised Cohen's class,
\begin{align*}
    Q_S T(w,z) := \langle T, \pi(z)S\pi(w)^*\rangle_{\mathcal{HS}},
\end{align*}
which is a quantum time--frequency representation of $T$ with respect to an appropriate window $S$. In the case of a rank--one $S=g\otimes g$, this becomes the (continuous) Gabor matrix \cref{gabormatrix}. Reconstruction from the discretised polarised Cohen's class follows in the same manner to the function case in time--frequency analysis, the so--called operator Gabor frames \cite{LuMc24}. The question of reconstructing an operator from a given subset of the Gabor matrix is much more delicate. In fact, even the injectivity of the map $T\mapsto \{Q_S T(\lambda,\lambda)\}_{\Lambda}$ for a given Feichtinger operator $S$ is an interesting problem. \\

In this work we employ tools from the symplectic and metaplectic groups when considering the problems of operator reconstruction and identification. The symplectic group $Sp(d,\mathbb{R})$ consists of the real $2d\times 2d$ matrices $\mathcal{A}$ satisfying
\begin{align*}
    \mathcal{A}^T J\mathcal{A} = J
\end{align*}
where
\begin{align*}
    J=\begin{pmatrix}
        0 && I_d \\
        -I_d && 0
    \end{pmatrix}
\end{align*}
is the standard skew--symmetric matrix. The symplectic and closely related metaplectic groups have a deep connection to time--frequency analysis, since the symplectic matrices and corresponding metaplectic operators are precisely those which intertwine the symmetric time--frequency shifts (cf. \cite{Go06}). Since most relevant time--frequency distributions are equivalent to one another via a metaplectic operator, it is natural to ask what statements particular to certain distributions in fact extend to other distributions related by a metaplectic operator, and how otherwise the metaplectic group interacts with central objects in TFA, leading to the recent work \cite{CoGi23} \cite{CoGi24}. We consider metaplectic operators acting on operators in a similar manner to the approach of \cite{Gi22} \cite{CoRo22}, and consider for which metaplectic operators we can answer the question of operator reconstruction and identification.

Since QHA and QTFA are formulated in terms of the Weyl symbol of an operator, we use the convention that a metaplectic operator \emph{acting on an operator}, corresponds to the metaplectic operator which acts on its Weyl symbol in the function sense. Namely, given some symplectic matrix $\mathcal{A}\in Sp(2d,\mathbb{R})$ and some operator $T\in\mathfrak{S}^\prime$, the \textit{$\mathcal{A}$-Weyl symbol} $\sigma^{\mathcal{A}}_T$ of $T$, is given by
\begin{align*}
    \sigma^{\mathcal{A}}_T = \mu(\mathcal{A})\sigma_T,
\end{align*}
and the metaplectic transform of $T$, denoted $\mu(\mathcal{A})T$, is the operator defined as 
\begin{align*}
    \mu(\mathcal{A})T := L_{\sigma^{\mathcal{A}}_T}
\end{align*}
where $L_\sigma$ is the Weyl quantisation of the symbol $\sigma$. Just as the most prominent time--frequency distributions are related by metaplectic operators \cite{CoGiRo24}, the operators arising from quantisations of these distributions are related by metaplectic transforms of operators. Using this metaplectic approach to QTFA, we show the following:
\begin{theorem}\label{intro:metadiagrecon}
    Let $\mathcal{A}\in Sp(2d,\mathbb{R})$ be an upper (block) triangular matrix, 
    \begin{align*}
    \mathcal{A} = 
        \begin{pmatrix}
           A_1 & A_2  \\
           0 & A_1^{-T}
        \end{pmatrix},
    \end{align*}
    where the $A_i$ are $2d \times 2d$ matrices, and let $h$, $\Lambda$ and $Q$ be as in \cref{intro:diagreconstr}, and $S\in\mathfrak{S}$ such that $\mathcal{F}_W(S)\neq 0$ in supp$(g)$. Then given some operator $T\in\mathfrak{S}^\prime$ satisfying 
    \begin{align*}
        \mathrm{supp}(\sigma_T^{J\mathcal{A}}) \subset (1-\epsilon)Q,
    \end{align*}
    $T$ can be expressed as
    \begin{align*}
        T = \sum_{\lambda\in\Lambda^\prime} Q_{S^\prime} T(\lambda) T_{\lambda} \alpha_\lambda (R)
    \end{align*}
    where $S^\prime = \mu(\mathcal{A})^*S$, $\mathcal{F}_W(R) = \mathcal{F}_{\Omega}\Big(\frac{h}{\mathcal{F}_W (S^\prime)}\Big)$ and $\Lambda^\prime = A_1\Lambda$.
\end{theorem}
Interestingly, the condition that $\mathcal{A}$ be of the type in \cref{intro:metadiagrecon} is equivalent to the operator $\mu(\mathcal{A})$ being bounded on all $M^{p,q}(\mathbb{R}^d)$ spaces \cite{FuSh24}.

We then relax the requirement that an operator be reconstructable from the diagonal of its Gabor matrix, and instead ask for which sets of operators the mapping $T\mapsto Q_S T(\lambda,\lambda)$ is injective. This is of course already the case in the setting of \cref{intro:diagreconstr} and \cref{intro:metadiagrecon}. We find the following:
\begin{theorem}
    Let $\mathcal{A}\in Sp(2d,\mathbb{R})$ have the block decomposition
    \begin{align}
        \begin{pmatrix}
            A_1 && A_2 \\
            0 && A_1^{-T}
        \end{pmatrix}
    \end{align}
    where $A_i$ are $2d\times 2d$ matrices. Then the space 
    \begin{align*}
        \mathcal{P}_{\mathcal{A},K} := \{S\in\mathcal{HS} : \mathrm{supp}\big(\sigma^{\mathcal{A}}_S\big) \subset K\}
    \end{align*}
    is identifiable on \emph{any} lattice $\Lambda\subset\mathbb{R}^{2d}$, with respect to the window $S=\mu(\mathcal{A})^*\left(\varphi_A\otimes \varphi_A\right)$, where $\varphi_A$ is the generalised Gaussian  
    \begin{align*}
        \varphi_A(t) := 2^{d/4}e^{-\pi t At}
    \end{align*}
    for some symmetric $A\in GL(d,\mathbb{R})$.
\end{theorem}
We then consider decompositions of symplectic matrices of the type 
\begin{align}\label{intro:metadecomp}
    \mathcal{C} = D_L V_{\Tilde{A}} V_{\Tilde{B}}^T V_A,
\end{align}
where symplectic matrices of the type $D_L$ and $V_A$ are the generators of the symplectic group, to be defined in \cref{prelims:symp}. In contrast to some of the decompositions of symplectic matrices, such a decomposition has the advantage of being easily calculable (Algorithm 1, \cite{JiTaZh20}). For such a decomposition, we have the following result:
\begin{theorem}
    Let $\mathcal{A}\in Sp(2d,\mathbb{R})$, such that $A,\Tilde{B}$ corresponding to $\mathcal{A}$ in the decomposition \cref{intro:metadecomp} are simultaneously diagonalisable. Then for any compact $K\subset\mathbb{R}^{2d}$, there exists some lattice $\Lambda = A\mathbb{Z}^{2d}$ such that $\mathcal{P}_{\mathcal{A},K}$ is identifiable on $\Lambda$.
\end{theorem}
As a particular example of this identification property, we consider the metaplectic transformation of operators mapping the Weyl symbol of an operator to the integral kernel, which gives the phase retrieval property for compactly supported functions, as considered in \cite{GrLi23} \cite{GrLiSh22}.

\section{Preliminaries}
\subsection{Harmonic and Time--Frequency Analysis}\label{intro:harmandtfa}
Harmonic analysis considers translations and Fourier translations, in the most general sense on locally compact abelian groups. The classical Fourier transform is defined as 
\begin{align*}
    \mathcal{F} f(\omega) := \int_{\mathbb{R}^d} f(x)e^{-2\pi i x\cdot \omega}\, dx,
\end{align*}
for a function $f\in L^1(\mathbb{R}^d)$. The Fourier transform can be extended to a unitary operator on $L^2(\mathbb{R}^d)$, and by duality to tempered distributions. The classical sampling result of Whittaker and Shannon, and the extension to the distributional case by Campbell, will be used to give reconstruction results for operators with band-limited symbols, as first considered in \cite{GrPa13}.
\begin{theorem}[Theorem 2, \cite{Ca68}]\label{shannsampling}
    Given a lattice $\Lambda \subset \mathbb{R}^{2d}$ with adjoint lattice $\Lambda^\circ = A\mathbb{Z}^{2d}$, let $Q$ denote the fundamental domain of $\Lambda^\circ$ given by $Q=A[-\tfrac{1}{2},\tfrac{1}{2}]^{2d}$. Let $f\in\mathscr{S}^\prime$ such that $\mathrm{supp}(\mathcal{F}_\Omega(f))\subset (1-\epsilon)Q$ for $0<\epsilon<1/2$. Let $h\in C^{\infty}_c(\mathbb{R}^{2d})$ such that $h|_{(1-\epsilon)Q}\equiv 1$ and $\mathrm{supp} (h) \subset Q$. Then
    \begin{align*}
        f = \frac{1}{|\Lambda|} \sum_{\lambda\in\Lambda} f(\lambda) T_{\lambda} \mathcal{F}_{\Omega}(h).
    \end{align*}
    where the sum converges in the distributional sense.
\end{theorem}
For non band-limited functions, reconstruction from discrete samples is in general not possible. However, by extending our perspective from harmonic analysis to time--frequency analysis, we find pleasing reconstruction results. The underlying aim of Time--Frequency Analysis (TFA) is to represent functions, or signals, in both time and frequency simultaneously, by analysing the function with respect to a translated and modulated atom or window function. We define the time--frequency shift $\pi(z)$ for some $z=(x,\omega)\in\mathbb{R}^{2d}$ by
\begin{align*}
    \pi(x,\omega) := M_\omega T_x,
\end{align*}
where $T_x f(t) = f(t-x)$ and $M_\omega f(t)= e^{2\pi i \omega\cdot t}f(t)$. The time--frequency shifts are a projective representation of the reduced Heisenberg group on the Hilbert space $L^2(\mathbb{R}^d)$, and accordingly many features of TFA in fact arise from the properties of the Heisenberg group. Closely related is the symmetric time--frequency shift $\rho(z)$, defined as
\begin{align*}
    \rho(z) := T_{x/2}M_\omega T_{x/2} = M_{\omega/2} T_x M_{\omega/2} = e^{-\pi i \omega\cdot x}\pi(z).
\end{align*}
After introducing the symplectic matrices and corresponding metaplectic operators, the symmetric time--frequency shifts will in many ways be the most natural to use. We have the following composition relations for the two shifts:
\begin{align*}
    \pi(z)\pi(z') &= e^{-2\pi i \omega'\cdot x} \pi(z+z'), \\
    \rho(z)\rho(z') &= e^{-\pi i \Omega(z,z')}\rho(z+z'),
\end{align*}
where $\Omega(z,z'):= \omega'\cdot x - \omega\cdot x'$ is the standard symmplectic form on $\mathbb{R}^{2d}$. The central object of time--frequency analysis, the Short--Time Fourier Transform (STFT) of $f$ with respect to window or atom $g$, can then be defined as
\begin{align*}
    V_g f(z) = \langle f, \pi(z)g\rangle_{L^2}
\end{align*}
for $f,g\in L^2(\mathbb{R}^d)$. The adjoint mapping is given by
\begin{align*}
    V_g^* F = \int_{\mathbb{R}^{2d}} F(z) \pi(z)g\, dz
\end{align*}
where the integral can be interpreted weakly. For $\|g\|_{L^2} = 1$, the mapping $f\mapsto V_g f$ is an isometry from $L^2(\mathbb{R}^d)$ to $L^2(\mathbb{R}^{2d})$, and moreover we have Moyal's orthogonality relation:
\begin{align*}
    \langle V_{g_1} f_1, V_{g_2} f_2 \rangle_{L^2(\mathbb{R}^{2d})} = \langle f_1, f_2 \rangle_{L^2(\mathbb{R}^d)} \overline{\langle g_1, g_2 \rangle}_{L^2(\mathbb{R}^d)},
\end{align*}
for $f_1,f_2,g_1,g_2\in L^2(\mathbb{R}^d)$. As a particular instance of Moyal's relation, the reconstruction formula is given by
\begin{align*}
    f = \frac{1}{\langle h,g\rangle_{L^2}} \int_{\mathbb{R}^{2d}} V_g f(z)\pi(z)h\, dz,
\end{align*}
for $\langle h, g\rangle_{L^2} \neq 0$. By using the symmetric time--frequency shift in place of the standard time--frequency shift, the (cross)-ambiguity function can be defined as 
\begin{align*}
    A(f,g)(z) := \langle f, \rho(z)g\rangle_{L^2} = e^{\pi i \omega\cdot x}V_g f(z).
\end{align*}
We also define the (cross)-Wigner distribution,
\begin{align*}
    W(f,g)(z) := \int_{\mathbb{R}^d} f\left( x+\frac{t}{2} \right)\overline{g\left( x-\frac{t}{2} \right)}e^{-2\pi i x\omega}\, dt.
\end{align*}
The Wigner distribution is related to the ambiguity function via the symplectic Fourier transform, If we define the symplectic Fourier transform as 
\begin{align*}
    \mathcal{F}_\Omega f(z) = \int_{\mathbb{R}^{2d}} f(z')e^{-2\pi i\, \Omega(z,z')}\, dz',
\end{align*}
we have the relation
\begin{align*}
    W(f,g) = \mathcal{F}_\Omega (A(f,g)).
\end{align*}
Conversely, taking the symplectic Fourier transform of the STFT gives the Rihaczek distribution;
\begin{align*}
    R(f,g)(z) &= \mathcal{F}_\Omega (V_g f)(z) \\
    &= f(x)\overline{\hat{g}(\omega)}e^{-2\pi i x\omega}.
\end{align*}

By restricting the analysing atom to the normalised Gaussian $\varphi_0(t) = 2^{d/4}e^{-\pi t^2}$ which is in the Schwartz space $\mathscr{S}(\mathbb{R}^d)$, the STFT can be defined in terms of the duality pairing $(\mathscr{S}^\prime,\mathscr{S})$ and extended to tempered distributions. The ambiguity function, Wigner distribution and Rihaczek distribution can be similarly extended. We can then consider the spaces of functions defined by the decay of their STFT, called the modulation spaces:
\begin{align*}
    M^{p,q}_m := \{ f \in\mathscr{S}^\prime(\mathbb{R}^d): V_{\varphi_0} f \in L^{p,q}_m(\mathbb{R}^{2d}) \}.
\end{align*}
The modulation spaces reflect the decay in both time and frequency, and in the special case $p=q=2$, $M^{2,2}(\mathbb{R}^{d}) = L^2(\mathbb{R}^{d})$. 

\subsection{Gabor Frames}
An attractive feature of time-frequency analysis is the possibility of discretisation of the central objects. For an appropriate choice of window, the STFT, sampled on a discrete subset of $\mathbb{R}^{2d}$, gives a basis like representation of a function, known as a frame. For some Hilbert space $\mathcal{H}$, a set $\{f_i\}_{i\in I}$ is called a frame if the frame condition
\begin{align*}
    A\|f\|_{\mathcal{H}}^2 \leq \sum_{i\in I} |\langle f, f_i\rangle_{\mathcal{H}}|^2 \leq B\|f\|_{\mathcal{H}}^2
\end{align*}
is satisfied for some constants $A,B$ and every $f\in\mathcal{H}$. Clearly, a basis is a frame with frame constants $A=B=1$, but the frame condition is more general, and the possible overcompleteness is in many cases desirable. The frame operator 
\begin{align*}
    Ef := \sum_{i\in I} \langle f, f_i\rangle_{\mathcal{H}} f_i
\end{align*}
is a positive, invertible operator on $\mathcal{H}$, and can be decomposed into the analysis operator $C:\mathcal{H}\to \ell^2(I)$;
\begin{align*}
    Cf:= \{\langle f, f_i\rangle_{\mathcal{H}}\}_{i\in I},
\end{align*}
and synthesis operator $D:\ell^2(I)\to\mathcal{H}$;
\begin{align*}
    Da := \sum_{i\in I}a_i f_i.
\end{align*}
The frames arising naturally in time--frequency analysis are of the type 
\begin{align*}
    \{\pi(\lambda)g\}_{\lambda\in\Lambda},
\end{align*}
where $\Lambda\subset\mathbb{R}^{2d}$ is a lattice, and are known as Gabor frames. The modulation spaces $M^{p,q}_m(\mathbb{R}^d)$ are characterised by their frame coefficients, given some Gabor frame $\{\pi(\lambda)g\}_{\lambda\in\Lambda}$ with $g\in M^1(\mathbb{R}^d)$, 
\begin{align*}
    f \in M^{p,q}_m(\mathbb{R}^d) \iff Cf \in \ell^{p,q}_m(\Lambda),
\end{align*}
and the norm of the analysis operator is independent of $p,q$ \cite{FeGr89i}.

\subsection{Quantisation Schemes}\label{intro:quantchemes}
The motivation for Quantum Time--Frequency Analysis lies in extending the tools of time--frequency analysis to operators. The trace of an operator $S\in\mathcal{L}(L^2(\mathbb{R}^d))$ is defined as
\begin{align*}
    \mathrm{tr}(S) := \sum_{n} \langle Se_n, e_n\rangle_{L^2}
\end{align*}
for some orthornomal basis $\{e_n\}_{n\in\mathbb{N}}$, upon which the trace does not depend. The trace class of operators, $\mathcal{S}^1$, is the space of operators with finite trace of their positive part;
\begin{align*}
    \mathcal{S}^1 := \{S\in\mathcal{L}(L^2(\mathbb{R}^d)): \mathrm{tr}(|S|) < \infty\}.
\end{align*}
$\mathcal{S}^1$ is a Banach space and ideal of $\mathcal{L}(L^2(\mathbb{R}^d))$. The Hilbert-Schmidt operators are defined as 
\begin{align*}
    \mathcal{HS} := \{S\in\mathcal{L}(L^2(\mathbb{R}^d)): S^*S\in\mathcal{S}^1\},
\end{align*}
and form a Hilbert space with respect to the inner product
\begin{align*}
    \langle S,T\rangle_{\mathcal{HS}} := \mathrm{tr}(ST^*).
\end{align*}
More generally, the Schatten class operators $\mathcal{S}^p$ can be defined by $\ell^p$ decay of their singular values. Trace class, Hilbert-Schmidt, and other Schatten class operators are subspaces of the compact operators, and accordingly admit a spectral decomposition,
\begin{align*}
    S = \sum_{n\in\mathbb{N}} \lambda_n \psi_n \otimes \phi_n,
\end{align*}
where $\lambda_n$ are the singular values of $S$, $\{\psi_n\}_{n\in\mathbb{N}}$ and $\{\phi_n\}_{n\in\mathbb{N}}$ are orthonormal sets in $L^2(\mathbb{R}^d)$ and the rank one operator $\psi \otimes \phi$ is defined by
\begin{align*}
    \psi \otimes \phi := \langle \cdot , \phi \rangle_{L^2} \psi.
\end{align*}
To an operator $S\in\mathcal{L}(L^2(\mathbb{R}^d))$, we can assign a function or distribution in several ways, corresponding to different quantisation schemes. The integral kernel of an operator $S$ is defined as the tempered distribution $K_S$ such that
\begin{align*}
    \langle Sf,g\rangle_{\mathscr{S}^\prime,\mathscr{S}} = \langle K_S,g\otimes f\rangle_{\mathscr{S}^\prime,\mathscr{S}}.
\end{align*}
The map $S\mapsto K_S$ is a unitary map from $\mathcal{HS}$ to $L^2(\mathbb{R}^{2d})$, and conversely any function $F\in L^2(\mathbb{R}^{2d})$ defines a Hilbert-Schmidt operator in this manner. In time--frequency analysis, a useful quantisation scheme is given by the Weyl quantisation. The Weyl symbol $\sigma_{f\otimes g}$ of a rank--one operator $f\otimes g$ is given by
\begin{align*}
    \sigma_{f\otimes g} = W(f,g),
\end{align*}
where $W(f,g)$ is the Wigner distribution introduced above, and this mapping can be extended linearly to the Hilbert-Schmidt operators due to the spectral decomposition. We denote the operator resulting from the Weyl quantisation of a symbol $\sigma$ as $L_\sigma$. The Weyl quantisation $\sigma \mapsto L_\sigma$ is also unitary from $L^2(\mathbb{R}^{2d})\to\mathcal{HS}$, and can be extended by duality to tempered distributions. The Weyl symbol of $S$, $\sigma_S$, is related unitarily to the integral kernel by 
\begin{align*}
    \sigma_S = \mathcal{F}_2 T_a K_S,
\end{align*}
where $\mathcal{F}_2$ is the partial Fourier transform in the second variable, and $T_a$ the change of coordinates $T_a:(x,y)\mapsto (x+\tfrac{y}{2},x-\tfrac{y}{2})$.

If instead of the Wigner distribution, we consider the quantisation scheme associated to the Rihaczek distribution, we recover the Kohn-Nirenberg quantisation. Namely, for a rank--one operator $f\otimes g$, the Kohn-Nirenberg symbol $\sigma_{f\otimes g}^{KN}$ is given by
\begin{align*}
    \kappa_{f\otimes g} = R(f,g).
\end{align*}
For all quantisation schemes presented, the quantisations of the space of Schwartz functions coincide, as do the quantisations of the tempered distributions. We denote these spaces of operators $\mathfrak{S}$ and $\mathfrak{S}^\prime$ respectively. Similarly, the modulation spaces $M^p(\mathbb{R}^{2d})$ correspond to the same space of operators for all quantisation schemes, which we denote $\mathcal{M}^p$.

\subsection{Quantum Time--Frequency Analysis}

The operator translation can be defined as
\begin{align*}
    \alpha_z (S) := \rho(z)S\rho(z)^*.
\end{align*}
This has the effect of translating the Weyl symbol;
\begin{align*}
    \sigma_{\alpha_z(S)} = T_z \sigma_S.
\end{align*}
If one then considers the trace of an operator as the analogue of the integral of a function, convolutions can be defined between operators and functions. Given $S\in\mathcal{S}^q$, $f\in L^p(\mathbb{R}^{2d})$ and $T\in\mathcal{S}^p$, for $\tfrac{1}{p}+\tfrac{1}{q} = 1+\tfrac{1}{r}$, we define
\begin{align*}
    S \star T(z) &:= \mathrm{tr}(S\alpha_z(\Check{T})) \\
    f \star S &:= \int_{\mathbb{R}^{2d}} f(z)\alpha_z(S)\, dz,
\end{align*}
where $\Check{T}=PTP$, and $P$ is the parity operator. We note that the convolution between two operators gives a function on phase space, while the convolution between an operator and a function gives an operator. The convolutions are commutative, and associative with the classical convolution between functions. In the case of rank--one operators, we recover familiar objects from time--frequency analysis. The convolution of a rank one operator $g\otimes h$ with a function $f\in L^2(\mathbb{R}^{2d})$, corresponds to the localisation operator;
\begin{align*}
    f\star (g\otimes h) \psi = \int_{\mathbb{R}^{2d}} f(z) V_{h} \psi(z) \pi(z)g \, dz.
\end{align*}
On the other hand, the convolution of two rank--one operators $g\otimes g$ and $f\otimes f$ is the spectrogram;
\begin{align*}
    (g\otimes g) \star (f\otimes f) = |V_g f(z)|^2.
\end{align*}

The Fourier-Wigner transform $\mathcal{F}_W$ is defined as
\begin{align*}
    \mathcal{F}_W (S) = \mathrm{tr}(\rho(-z)S),
\end{align*}
and is a unitary map from $\mathcal{HS}$ to $L^2(\mathbb{R}^{2d})$. We have the relation
\begin{align*}
    \mathcal{F}_\Omega(\mathcal{F}_W(S)) = \sigma_S.
\end{align*}
The operator convolutions satisfy the Fourier convolution property with respect to the Fourier-Wigner transform;
\begin{align}\label{fourierwigconv}
    \mathcal{F}_W(f\star S) &= \mathcal{F}_\Omega (f)\cdot \mathcal{F}_W (S)\nonumber \\
    \mathcal{F}_\Omega (S\star T) &= \mathcal{F}_W (S) \cdot \mathcal{F}_W (T).
\end{align}
As in the case of the convolutions, the Fourier-Wigner transform of a rank--one operator returns a familiar object, the ambiguity function:
\begin{align*}
    \mathcal{F}_W (g\otimes f)(z) = A(f,g)(z).
\end{align*}
Just as the translation of an operator is equivalent to a translation of its Weyl symbol, the modulation of an operator;
\begin{align*}
    \beta_w (S) := \rho(\tfrac{w}{2})S\rho(\tfrac{w}{2}),
\end{align*}
is equivalent to a modulation of its Weyl symbol, or a translation of its Fourier-Wigner transform. With these two operations defined, we introduce the time--frequency shift of an operator 
\begin{align*}
    \gamma_{w,z}(S) = e^{-\pi i (z_1 z_2 - w_1 w_2)}\rho(z)S\rho(w)^*
\end{align*}
and the corresponding polarised Cohen's class of an operator $T\in\mathcal{HS}$ with respect to a window operator $S\in\mathcal{HS}$:
\begin{align*}
    Q_S T(w,z) = \langle T, \gamma_{w,z}(S)\rangle_{\mathcal{HS}}.
\end{align*}
The polarised Cohen's class is an isometric mapping $Q_S:\mathcal{HS}\to L^2(\mathbb{R}^{4d})$, and satisfies the orthogonality property; 
\begin{align*}
    \langle Q_S T, Q_R W\rangle_{L^2} = \langle T, W\rangle_\mathcal{HS}\overline{\langle S, R\rangle_\mathcal{HS}}, 
\end{align*}
and reproducing formula;
\begin{align*}
    T = \frac{1}{\|S\|_{\mathcal{HS}}^2} \int_{\mathbb{R}^{4d}} Q_S T(w,z) \gamma_{w,z}(S)\, dw\, dz.
\end{align*}
Since the operator time--frequency shift $\gamma_{w,z}$ corresponds to a time--frequency shift of the Weyl symbol, we find the relation
\begin{align*}
    Q_S T(w,z) = c_{w,z}V_{\sigma_S} \sigma_T (U(w,z)),
\end{align*}
where $c_{w,z}$ is the unitary phase factor
\begin{align*}
    c_{w,z} = e^{-i\pi(w_1+z_1)(w_2-z_2)}
\end{align*}
and $U$ the change of variables
\begin{align*}
    U(w,z) = \left( \frac{w_1+z_1}{2}, \frac{w_2+z_2}{2}, w_2-z_2, z_1-w_1 \right).
\end{align*}
As in the function case, by restricting the window $S$ to the space of operators with symbols in the Schwartz class, $\mathfrak{S}$, we can extend the polarised Cohen's class to operators with Weyl symbols in the space of tempered distributions, $\mathfrak{S}^\prime$. The operator modulation spaces $\mathcal{M}^{p,q}$ are then the spaces defined for $1\leq p,q \leq \infty$ by
\begin{align*}
    \mathcal{M}^{p,q}_m := \{T\in \mathfrak{S}': Q_{S_0} T \in L^{p,q}_m(\mathbb{R}^{4d})\}.
\end{align*}
where $S_0 = \varphi_0 \otimes \varphi_0$.

As in the case of classical time--frequency analysis, a crucial tool of quantum time--frequency analysis is the discretisation properties of the polarised Cohen's class. An operator Gabor frame is a frame in $\mathcal{HS}$ of the form
\begin{align*}
    \{\gamma_{\lambda,\mu}(S)\}_{(\lambda,\mu)\in\Lambda\times M}.
\end{align*}
The analysis operator, $C_S:\mathcal{HS}\to \ell^2(\Lambda\times M)$, corresponding to an operator Gabor frame, is given by
\begin{align*}
    C_S T := \{Q_S T(\lambda,\mu)\}_{\Lambda\times M}.
\end{align*}
Its adjoint, the synthesis operator $D_S:\ell^2(\Lambda\times M)\to\mathcal{HS}$ is given by 
\begin{align*}
    D_S a := \sum_{\Lambda\times M} a_{\lambda,\mu}\gamma_{w,z}(S).
\end{align*}
The composition of the two is defined as the frame operator, $E_{S,T}:=D_T C_S$, and the Gabor frame condition is equivalent to the frame operator being bounded and invertible. Given a window $S\in\mathcal{M}^1$ which generates a Gabor frame for $\mathcal{HS}$ on $\Lambda\times M$, the frame operator $E_{S,S}$ is bounded on all $\mathcal{M}^{p,q}$ spaces \cite{LuMc24}.

\subsection{Symplectic Matrices and Metapletic Operators}\label{prelims:symp}
The symplectic matrix $J$ is defined as
\begin{align*}
    J=\begin{pmatrix}
        0 && I_d \\
        -I_d && 0
    \end{pmatrix},
\end{align*}
where $I_d$ is the $d\times d$ identity matrix. In the sequel, we write simply $I$ when the dimension is clear. The symplectic form introduced in \cref{intro:harmandtfa} can hence be written as $\Omega(z,z')=z\cdot Jz'$. The symplectic group $Sp(d,\mathbb{R})$ consists of the matrices $\mathcal{A} \in \mathbb{R}^{2d\times 2d}$ satisfying
\begin{align*}
    \mathcal{A}^T J \mathcal{A} = J.
\end{align*}
If we write $\mathcal{A}$ in block form as
\begin{align*}
    \mathcal{A} = \begin{pmatrix}
        A && B \\
        C && D
    \end{pmatrix},
\end{align*}
then the symplectic condition is equivalent to all of the following holding:
\begin{align*}
    AC^T &= A^T C\\
    BD^T &= B^T D \\
    A^TD &- C^TB = I.
\end{align*}
For all $\mathcal{A}\in Sp(d,\mathbb{R})$, det$(\mathcal{A})=1$ and the inverse is given by
\begin{align}\label{metaplecticinverse}
    \mathcal{A}^{-1} = \begin{pmatrix}
        D^T && -B^T \\
        -C^T && A^T
    \end{pmatrix}.
\end{align}
If det$(B)\neq 0$, $\mathcal{A}$ is called a free symplectic matrix. The symplectic group is a subgroup of the special linear group, and the two coincide only in the case $d=1$. We define the $2d\times 2d$ matrices
\begin{align*}
    V_C := \begin{pmatrix}
        I && 0 \\
        C && I
    \end{pmatrix} \quad \mathrm{ and } \quad  D_L := \begin{pmatrix}
        L^{-1} && 0 \\
        0 && L^T
    \end{pmatrix},
\end{align*}
where $C\in\mathbb{R}^{d\times d}$ is symmetric and $L\in GL(d,\mathbb{R})$ non-singular. Along with $J$, matrices of the type $V_C$ and $D_L$ generate the symplectic group. It is not difficult to calculate that
\begin{align*}
    V_{C_1}V_{C_2} = V_{C_1+C_2} \quad \mathrm{ and } \quad D_{L_1}D_{L_2} = D_{L_2 L_1}.
\end{align*}
The symplectic group is a classical lie group, and it admits a double cover, the metaplectic group. Recalling the composition relation for the symmetric time--frequency shifts;
\begin{align*}
    \rho(z)\rho(z') = e^{-\pi i zJz'}\rho(z+z'),
\end{align*}
symplectic matrices are therefore defined by the property 
\begin{align*}
\rho(\mathcal{A}z)\rho(\mathcal{A}z') = e^{-\pi i zJz'}\rho(\mathcal{A}(z+z')).
\end{align*}
By the Stone-von Neumann theorem, the representation $z\mapsto \rho(\mathcal{A}z)$ is unitarily equivalent to the representation given by $\rho(z)$. We call this unitary operator the metaplectic operator corresponding to $\mathcal{A}$, denoted $\mu(\mathcal{A})$, which hence satisfies
\begin{align}\label{metaplecticintertwining}
    \mu(\mathcal{A})\rho(z)\mu(\mathcal{A})^* = \rho(\mathcal{A}z).
\end{align}

The operator is unique up to a unitary constant. We work with the particular choice of phase such that the generating symplectic matrices correspond to the following metaplectic operators for a function $f\in L^2(\mathbb{R}^d)$:
\begin{align*}
    \mu(J)f(x) &= \mathcal{F} f(x) \\
    \mu(V_C)f(x) &= e^{i\pi x\cdot Ct}f(x) \\
    \mu(D_L)f(x) &= |L|^{-1/2}f(Lx).
\end{align*}
We also note that 
\begin{align*}
    \mu(V_C^T)f(x) &= \widehat{e^{i\pi x\cdot Ct}} * f(x).
\end{align*}
\begin{example}\normalfont
    As a simple illustration of the intertwining property, consider the case $\mathcal{A}=J$, recalling that $\mu(J)$ is the Fourier transform. The identity 
    \begin{align*}
        V_g f(x,\omega) = e^{-2\pi i x\omega}V_{\Hat{g}} \Hat{f}(\omega,-x)
    \end{align*}
    for $f,g\in L^2(\mathbb{R}^d)$ can easily be calculated directly, however we can also use \cref{metaplecticintertwining} to verify
    \begin{align*}
        \langle \mu(J)f, \pi(Jz)\mu(J)g\rangle_{L^2} &= e^{\pi i x\omega}\langle f, \mu(J)^*\rho(Jz)\mu(J)g\rangle_{L^2} \\
        &= e^{\pi i x\omega}\langle f, \mu(J^{-1})\rho(Jz)\mu(J^{-1})^*g\rangle_{L^2} \\
        &= e^{\pi i x\omega}\langle f, \rho(z)g\rangle_{L^2} \\
        &= e^{2\pi i x\omega}\langle f, \pi(z)g\rangle_{L^2}
    \end{align*}
    where we use that $\mu$ is group homomorphism. 
\end{example}
The symplectic Fourier transform on $L^2(\mathbb{R}^{2d})$ corresponds to the symplectic matrix $\mathcal{A}_{\mathcal{F}_\Omega}\in Sp(2d,\mathbb{R})$ given by 
\begin{align*}
    \mathcal{A}_{\mathcal{F}_\Omega} = \begin{pmatrix}
        0 && 0 && 0 && -I \\
        0 && 0 && I && 0 \\
        0 && I && 0 && 0 \\
        -I && 0 && 0 && 0
    \end{pmatrix},
\end{align*}
while the partial Fourier transform in the second variable corresponds to 
\begin{align*}
    \mathcal{A}_{\mathcal{F}_2} = \begin{pmatrix}
        I && 0 && 0 && 0 \\
        0 && 0 && 0 && I \\
        0 && 0 && I && 0 \\
        0 && -I && 0 && 0
    \end{pmatrix}.
\end{align*}
Due to the intertwining property \cref{metaplecticintertwining}, the metaplectic group has a deep connection to time--frequency analysis, and many familiar objects are related to one another by metaplectic operators. 

Due to its role in quantum time--frequency analysis, the Wigner distribution will play a central role in this work. Recalling that the ambiguity function is given by the symplectic Fourier transform of the Wigner distribution, we have that 
\begin{align*}
    A(f,g) = \mu(\mathcal{A}_{\mathcal{F}_\Omega})W(f,g).
\end{align*}
Furthermore, since the STFT is simply the ambiguity function multiplied by a chirp;
\begin{align*}
    V_g f &= \mu(V_{C_0})\mu(\mathcal{A}_{\mathcal{F}_\Omega})W(f,g) \\
     &= \mu(\mathcal{A}_{STFT})W(f,g),
\end{align*}
where 
\begin{align*}
    C_0 = \begin{pmatrix}
        0 && -I/2 \\
        -I/2 && 0 \\
    \end{pmatrix},
\end{align*}
we have
\begin{align*}
    \mathcal{A}_{STFT} = \begin{pmatrix}
        0 && 0 && 0 && -I \\
        0 && 0 && I && 0 \\
        0 && I && -I/2 && 0 \\
        -I && 0 && 0 && I/2
    \end{pmatrix}.
\end{align*}
Finally since the Rihaczek distribution is given by $R(f,g)=\mathcal{F}_\Omega (V_g f)$, we have $R(f,g)=\mu(\mathcal{A}_{Rih})W(f,g)$, where
\begin{align*}
    \mathcal{A}_{Rih} = \begin{pmatrix}
        I && 0 && 0 && -I/2 \\
        0 && I && -I/2 && 0 \\
        0 && 0 && I && 0 \\
        0 && 0 && 0 && I
    \end{pmatrix}.
\end{align*}

\section{Discrete Reconstruction of Underspread Operators}
We begin by considering the reconstruction of an operator by use of operator Gabor frames \cite{LuMc24}. Recall that an operator $S\in\mathcal{HS}$ is said to generate a Gabor frame on the lattice $\Lambda\times M$ if the frame operator 
\begin{align*}
    E_S T :&= \sum_{\Lambda\times M} \langle T, \gamma_{\lambda,\mu}(S)\rangle_{\mathcal{HS}} \gamma_{\lambda,\mu}(S) \\
    &= \sum_{\Lambda\times M} Q_S T(\lambda,\mu) \gamma_{\lambda,\mu}(S)
\end{align*}
is bounded and invertible on the space $\mathcal{HS}$. Consider now the rank--one operator $S=g\otimes g$. The polarised Cohen's class is then given by
\begin{align}\label{gabormatrixdef}
    Q_S T(w,z) &= \langle T, \gamma_{\lambda,\mu}(S)\rangle_{\mathcal{HS}} \nonumber \\
    &= e^{\pi i (z_1 z_2 - w_1 w_2)}\langle T\rho(z)g, \rho(w)g\rangle_{L^2}.
\end{align}
Hence, the Gabor matrix 
\begin{align*}
    \left(e^{\pi i (\mu_1\mu_2 - \lambda_1\lambda_2)}\langle T\rho(\mu)g, \rho(\lambda)g\rangle_{L^2}\right)_{\Lambda\times M}
\end{align*}
corresponds to the polarised Cohen's class with respect to the rank--one operator up to a phase factor, sampled on a lattice. Since the tensor product of Gabor frames for functions give Gabor frames for operators \cite{Ba08}, if we take some $g\in L^2(\mathbb{R}^d)$ which generates a Gabor frame on $\Lambda$, then $S=g\otimes g$ generates a Gabor frame for operators on $\Lambda\times \Lambda$. Then reconstruction from the Gabor matrix \cref{gabormatrixdef} follows from the frame condition. Furthermore, if we take $g\in M^1(\mathbb{R}^d)$ and again set $S=g\otimes g$, then the frame operator $E_S$ is a bounded operator on all $\mathcal{M}^{p,q}$ spaces, and accordingly, any operator $T\in \mathcal{M}^{p,q}$ can be reconstructed from the Gabor matrix
\begin{align*}
    \left( Q_{g\otimes g} T(\lambda,\mu) \right) = \left(e^{\pi i (\mu_1\mu_2 - \lambda_1\lambda_2)}\langle T\rho(\mu)g, \rho(\lambda)g\rangle_{\mathscr{S}^\prime,\mathscr{S}}\right)_{\Lambda\times M}.
\end{align*}
Reconstruction from the Gabor matrix is thus possible for a large class of operators; the space $\mathcal{M}^\infty$ contains for example all bounded operators on $L^2(\mathbb{R}^d)$. Restricting to the diagonal $\left( Q_S T(\lambda,\lambda) \right)_{\Lambda}$, on the other hand, drastically reduces the possibilities for reconstruction, or even identification.
\\
The problem of identifying operators from the diagonal of the Gabor matrix,
\begin{align*}
    Q_{g\otimes g} S(\lambda,\lambda) = \langle S\rho(\lambda)g,\rho(\lambda)g\rangle
\end{align*}
for some lattice $\Lambda\subset \mathbb{R}^{2d}$ is known as channel estimation, and is an important objective in wireless communication. If we restrict our consideration to rank--one self--adjoint operators, channel estimation is the problem of phase retrieval on the lattice $\Lambda$, since if $S=f\otimes f$, then
\begin{align*}
    \langle S\rho(\lambda)g,\rho(\lambda)g\rangle &= \langle f,\rho(\lambda)g\rangle\cdot\overline{\langle f,\rho(\lambda)g\rangle} \\
    &= |A(f,g)(\lambda)|^2 \\
    &= |V_g f(\lambda)|^2.
\end{align*}
Since discrete translates on a lattice of a Hilbert-Schmidt operator can never be dense in the space of Hilbert-Schmidt operators (Proposition 7.2, \cite{Sk20}), we cannot hope for a general frame of the form $\{\alpha_\lambda(S)\}_{\lambda\in\Lambda}$ for a lattice $\Lambda\subset\mathbb{R}^{2d}$, and the discretisation $\Theta: T\mapsto \{Q_S T(\lambda,\lambda)\}_{\lambda\in\Lambda}$ cannot be injective as a map from the space of Hilbert-Schmidt operators. Moreover, even restricting the problem to the rank--one setting, we can never have a lattice for which the mapping of a function to the samples of its spectrogram is injective \cite{GrLi22}.

Instead, we will in this work consider for which sets of operators the mapping of an operator to the diagonal of its channel matrix is injective, that is to say for which sets of operators we can distinguish the operators based on the diagonal of their Gabor matrix. We begin by considering the well known case of discrete reconstruction for operators with compactly supported Fourier-Wigner transform, also known as underspread operators.

Recognising the diagonal of the Gabor matrix as the operator convolution, and hence a convolution of Weyl symbols, means one can use sampling theory for Paley-Wiener spaces. This was presented originally by Gröchenig and Pauwels \cite{GrPa13}, and in the language of quantum harmonic analysis in \cite{Sk20}. We recall some of these results in the language of quantum time--frequency analysis, which will serve as the foundation of generalisations in the sequel. 

\begin{theorem}[Proposition 2, \cite{GrPa13}, Theorem 7.4, \cite{Sk20}]\label{diagreconstr}
    Let $\Lambda, Q$ and $h$ be as in \cref{shannsampling}, and $S\in\mathfrak{S}$ such that $\mathcal{F}_W (S)$ is non-zero in $\mathrm{supp}(h)$. Then given $T\in\mathfrak{S}^\prime$ with $\mathcal{F}_W(T)\subset (1-\varepsilon)Q$ where $0<\varepsilon<1/2$;
    \begin{align*}
        T = \sum_{\lambda\in\Lambda} Q_S T(\lambda,\lambda) \alpha_\lambda (R),
    \end{align*}
    where $\mathcal{F}_W(R)=\frac{h}{|\Lambda|\cdot \mathcal{F}_W(\Check{S})}$.
\end{theorem}
\begin{proof}
    By \cref{fourierwigconv}, we have
    \begin{align}\label{fourconvidentity}
        \mathcal{F}_{\Omega} (Q_S T(w,w))(z) = \mathcal{F}_{W}(\Check{S})(z)\cdot \mathcal{F}_W (T)(z),
    \end{align}
    and so $\mathrm{supp}\big(\mathcal{F}_{\Omega}(Q_S T(w,w))\big) \subset (1-\varepsilon)Q$. Using \cref{shannsampling} then gives
    \begin{align*}
        Q_S T(z,z) &= \frac{1}{|\Lambda|}\sum_{\lambda\in\Lambda} Q_S T(\lambda,\lambda) T_\lambda \mathcal{F}_\Omega (h)(z),
    \end{align*}
    which upon taking the symplectic Fourier transform becomes
    \begin{align*}
        \mathcal{F}_\Omega (Q_S T(w,w))(z,z) &= \frac{1}{|\Lambda|}\sum_{\lambda\in\Lambda} Q_S T(\lambda,\lambda) e^{2\pi i \Omega(\lambda,\cdot)}h(z).
    \end{align*}
    Substituting \cref{fourconvidentity} then gives
    \begin{align*}
        \mathcal{F}_W (T) &= \frac{1}{|\Lambda|}\sum_{\lambda\in\Lambda} Q_S T(\lambda,\lambda) e^{2\pi i \Omega(\lambda,\cdot)}\frac{h}{\mathcal{F}_W(\Check{S})} \\
        &= \sum_{\lambda\in\Lambda} Q_S T(\lambda,\lambda) \mathcal{F}_W (\alpha_\lambda(R)),
    \end{align*}
    for $R$ as defined in the initial statement. Hence since the Fourier-Wigner is an isomorphism from $\mathfrak{S}^\prime$ to $\mathscr{S}^\prime(\mathbb{R}^{2d})$, 
    \begin{align*}
        T = \sum_{\lambda\in\Lambda} Q_S T(\lambda,\lambda) \alpha_\lambda (R).
    \end{align*}
    
\end{proof}
\cref{diagreconstr} essentially reduces to a statement on the convolution of Weyl symbols. Accordingly, we can also consider a periodic Fourier-Wigner, as opposed to compactly supported. Such operators can be defined as $\Lambda^{\circ}$-modulation invariant operators:
\begin{definition}
    An operator $T\in\mathcal{M}^{\infty}$ is called \text{$\Lambda$-modulation-invariant} if 
    \begin{align*}
        T\rho\Big(\frac{\lambda}{2}\Big) = \rho\Big(\frac{\lambda}{2}\Big)^*T
    \end{align*}
    for every $\lambda \in \Lambda$. 
\end{definition}
It is an easy calculation to find that such operators have the property
\begin{align*}
    \mathcal{F}_W(T)(z-\lambda) = \mathcal{F}_W(T)(z)
\end{align*}
for all $\lambda\in\Lambda$. Consequently, the Weyl symbol is supported on the lattice $\Lambda$, and the following result follows:
\begin{proposition}[Theorem 6.5, \cite{LaMc24}]
    Let $S\in\mathcal{M}^1$ such that $\sigma_S(0)\neq 0$, and  $\mathrm{supp}(\sigma_S)\subset K$, where $K$ is a compact neighbourhood of the origin not containing any $\lambda^\circ\in\Lambda^\circ$ apart from the origin. Then given any $\Lambda$-modulation invariant operator $T\in\mathcal{M}^{\infty}$; 
    \begin{align*}
        T = \frac{1}{\sigma_S(0)}\sum_{\lambda^\circ\in\Lambda^\circ} Q_S T(\lambda^\circ,\lambda^\circ)\cdot \rho(2\lambda^\circ)P.
    \end{align*}
\end{proposition}
\begin{remark}
    \normalfont Since the condition on the support of $S$ is impossible to satisfy for some rank--one, or indeed finite rank operator, the above result relies on the polarised Cohen's class interpretation of the Gabor matrix.
\end{remark}
From the definition of the Gabor matrix, the side diagonals $\left(Q_S T(\lambda+\eta,\lambda)\right)_{\lambda\in\Lambda}$ actually correspond to the diagonal of the polarised Cohen's class, with the new window $\rho(\eta)S$ (with the addition of a phase factor). As such, reconstruction of $T$ from the side diagonal is a simple corollary of \cref{diagreconstr}:
\begin{corollary}[Theorem 4.3.4, \cite{Pau11}]\label{sidediagrecon}
    Let $g$, $\Lambda$, $Q$ and $T$ be as in \cref{diagreconstr}. Given $S\in \mathfrak{S}$ such that $\mathrm{supp} (\mathcal{F}_W(\rho(\eta)S)) \subset (1-\epsilon)Q$ for some $\eta\in\Lambda$ and $\mathcal{F}_W (S)$ is non-zero in $\mathrm{supp}(g)$, one can reconstruct $T$ from the side diagonal
    \begin{align*}
        T = \sum_{\lambda\in\Lambda} e^{-2\pi i \eta_2(\lambda_1+\eta_1)}Q_S T(\lambda,\lambda+\eta) \alpha_{\lambda} (R),
    \end{align*}
    where $\mathcal{F}_W(R) = \frac{g}{|\Lambda|\cdot\mathcal{F}_W (\rho(\eta)\Check{S})}$.
\end{corollary}

\section{Metaplectic Wigner distributions and operator reconstruction}
To extend reconstruction property statements from the previous section, we consider metaplectic transformations of operators. In particular we will show that reconstruction can be extended to any operator whose image under a certain type of metaplectic transform has a compactly supported Fourier-Wigner distribution. To begin with, we define metaplectic transforms of operators:
\begin{definition}
    Given some symplectic matrix $\mathcal{A}\in Sp(2d,\mathbb{R})$ and some operator $T\in\mathfrak{S}^\prime$, the \textit{$\mathcal{A}$-Weyl symbol} $\sigma^{\mathcal{A}}_T$ of $T$, is given by
    \begin{align*}
        \sigma^{\mathcal{A}}_T = \mu(\mathcal{A})\sigma_T.
    \end{align*}
    The metaplectic transform of $T$, denoted $\mu(\mathcal{A})T$ is the operator defined as 
    \begin{align*}
        \mu(\mathcal{A})T := L_{\sigma^{\mathcal{A}}_T}
    \end{align*}
    where $\mu(\mathcal{A})\sigma$ is the usual metaplectic operator $\mu(\mathcal{A})$ acting on the distribution $\sigma\in \mathscr{S}^\prime(\mathbb{R}^{2d})$. Hence the metaplectic transform $\mu(\mathcal{A})T$ is the Weyl quantisation of the metaplectic transform of the Weyl symbol of $T$. Conversely, the $\mathcal{A}$-quantisation of some $\sigma\in\mathscr{S}^\prime$ is the operator $L^{\mathcal{A}}_\sigma$ defined as 
    \begin{align*}
        L^{\mathcal{A}}_\sigma := L_{\mu(\mathcal{A})^*\sigma}.
    \end{align*}
\end{definition}

\begin{example}\label{symbkerntrans}\normalfont
    For 
    \begin{align*}
        \mathcal{A} = \begin{pmatrix}
                    \frac{1}{2} & \frac{1}{2} & 0 & 0 \\
                    0 & 0 & \frac{1}{2} & -\frac{1}{2} \\
                    0 & 0 & 1 & 1 \\
                    -1 & 1 & 0 & 0 \\                   
                 \end{pmatrix},
    \end{align*}
    $\mu(\mathcal{A})$ defines the mapping $\sigma_T \mapsto k_T$ (cf. \cite{Gi22}), and so the $\mathcal{A}$-Weyl symbol is the kernel of $T$. On the other hand the $\mathcal{A}$-quantisation of the kernel $K_T$ gives the operator $T$.
\end{example}

\begin{remark}\normalfont
    One should note that in \cite{CoGi23}, $\mu(\mathcal{A})(T)$, is used to refer to $\mu(\mathcal{A})k_T$, where $k_T$ is the integral kernel of $T$. This definition is related to ours by the transformation of the symplectic matrix
    \begin{align*}
        \Tilde{\mathcal{A}} = \begin{pmatrix}
            1 & 0 & 0 & -\frac{1}{2} \\
            1 & 0 & 0 & \frac{1}{2} \\
            0 & 1 & \frac{1}{2} & 0 \\
            0 & -1 & \frac{1}{2} & 0 \\                   
         \end{pmatrix} \cdot \mathcal{A} \cdot\begin{pmatrix}
            \frac{1}{2} & \frac{1}{2} & 0 & 0 \\
            0 & 0 & \frac{1}{2} & -\frac{1}{2} \\
            0 & 0 & 1 & 1 \\
            -1 & 1 & 0 & 0 \\             
        \end{pmatrix}
    \end{align*}
    by \cref{symbkerntrans}. We choose the convention in the definition in terms of the Weyl symbol to lend more intuition to certain results in the sequel.
\end{remark}
Recalling that Weyl quantisation can be weakly defined by
\begin{align*}
    \langle L_\sigma f, g\rangle = \langle \sigma, \sigma_{f\otimes g} \rangle,
\end{align*}
we note that the $\mathcal{A}$-Weyl symbol and $\mathcal{A}$-quantization are dual in the sense that
\begin{align*}
    \langle L^{\mathcal{A}}_\sigma f, g\rangle = \langle \sigma, \sigma^\mathcal{A}_{f\otimes g} \rangle.
\end{align*}
We recall the intertwining property of metaplectic transforms \cref{metaplecticintertwining};
    \begin{align*}
        \mu(\mathcal{A})\rho(\xi)\mu(\mathcal{A})^* = \rho(\mathcal{A}\xi)
    \end{align*}
for $\mathcal{A}\in Sp(2d,\mathbb{R})$.

We are now ready to state the reconstruction theorem for $\mathcal{A}$-Weyl symbols:
\begin{theorem}\label{metadiagrecon}
    Let $\mathcal{A}\in Sp(2d,\mathbb{R})$ be an upper triangular matrix, 
    \begin{align}
    \mathcal{A} = 
        \begin{pmatrix}
           A_1 & A_2  \\
           0 & A_1^{-T}
        \end{pmatrix},
    \end{align}
    where the $A_i$ are $2d \times 2d$ matrices, and let $g$, $\Lambda$, $Q$ and $S$ be as in \cref{diagreconstr}. Then given some operator $T\in\mathfrak{S}^\prime$ satisfying 
    \begin{align*}
        \mathrm{supp}(\sigma_T^{J\mathcal{A}}) \subset (1-\epsilon)Q,
    \end{align*}
    $T$ can be expressed as
    \begin{align*}
        T = \sum_{\lambda\in\Lambda^\prime} Q_{S^\prime} T(\lambda,\lambda) T_{\lambda} \alpha_\lambda (R)
    \end{align*}
    where $S^\prime = \mu(\mathcal{A})^*S$, $\mathcal{F}_W(R) = \mathcal{F}_{\Omega}\Big(\frac{g}{\mathcal{F}_W (\Check{S^\prime})}\Big)$ and $(\Lambda^\prime\times M) = \mathcal{A}^{-1}(\Lambda\times 0)$.
\end{theorem}
\begin{proof}
    By \cref{diagreconstr} and the definition of $\sigma_T^{J\mathcal{A}}$, we have that the operator $\mu(\mathcal{A})T$ can be reconstructed as
    \begin{align*}
        \mu(\mathcal{A})T = \sum_{\lambda\in\Lambda} Q_S \big(\mu(\mathcal{A})T\big)(\lambda,\lambda)\alpha_\lambda (R),
    \end{align*}
    where $R=\frac{g}{|\Lambda|\cdot\mathcal{F}_W (\Check{S})}$. We then use the intertwining property \cref{metaplecticintertwining} to find
    \begin{align*}
        Q_S \big(\mu(\mathcal{A})T\big)(\lambda,\lambda) &= \langle \mu(\mathcal{A})T, \alpha_\lambda (S)\rangle_{}  \\
        &= \langle \mu(\mathcal{A})\sigma_T, \rho(\lambda,0)\sigma_S \rangle \\
        &= \langle \sigma_T, \mu(\mathcal{A})^*\rho(\lambda,0)\mu(\mathcal{A})\mu(\mathcal{A})^*\sigma_S \rangle \\
        &= \langle \sigma_T, \rho\big(\mathcal{A}^{-1}(\lambda,0)\big)\mu(\mathcal{A})^*\sigma_S \rangle \\
        &= Q_{S^\prime} T\big(\mathcal{A}^{-1}(\lambda,0)\big).
    \end{align*}
    Sampling $Q_S \mu(\mathcal{A})T$ then corresponds to sampling $Q_{S^\prime} T$ on the symplectic transformation of the lattice $\Lambda$. In order to be able to reconstruct $T$ from the diagonal of $Q_{S^\prime}$, translations by the lattice $\Lambda$ must correspond to translations on the lattice $\Lambda^\prime$, which is to say we need the condition
    \begin{align*}
        \mathcal{A}^{-1}(\lambda,0) = (\lambda^\prime,0).
    \end{align*}
    This condition is equivalent to
    \begin{align*}
    \mathcal{A}^{-1} = 
        \begin{pmatrix}
           A_1 & A_2  \\
           0 & A_3
        \end{pmatrix}.
    \end{align*}
    As a symplectic matrix we have the explicit form of the inverse \cref{metaplecticinverse};
    \begin{align*}
    \mathcal{A} = 
        \begin{pmatrix}
           A_3^T & -A_2^T  \\
           0 & A_1^T
        \end{pmatrix}.
    \end{align*}
    The upper block diagonal form means that for $\mathcal{A}$ to be symplectic, $A_3$ must be invertible, and so by relabelling the submatrices and using the definition of symplectic matrices, the result follows.
    
\end{proof}
\begin{example}\normalfont
    In \cite{GrPa13}, the reconstruction result \cref{diagreconstr} is developed in terms of Kohn-Nirenberg symbols and corresponding Rihaczek distributions, while in \cite{Sk20} the Fourier-Wigner and corresponding Wigner distribution are considered. These can be seen to be equivalent using \cref{metadiagrecon}. The map $U: \sigma_S \mapsto \sigma_S^{KN} $ is given by 
    \begin{align*}
        U \phi(x,y) = \mathcal{F}^{-1} e^{i\pi x\cdot y} \mathcal{F}\phi(x,y).
    \end{align*}
    In terms of metaplectic transformations, this is equivalent to 
    \begin{align*}
        U = \mu(J)^*\mu(V_C)\mu(J),
    \end{align*}
    where $C=\begin{pmatrix}
                    0 & \frac{1}{2}\\
                    \frac{1}{2} & 0\\                   
                 \end{pmatrix} $
    , and $V_C = \begin{pmatrix}
                    I & 0\\
                    C & I\\                   
                 \end{pmatrix}$ (Section 2.2, \cite{CoRo22}). 
    Hence the matrix $\mathcal{A}$ in the previous theorem is in this case
        \begin{align*}
         J^{T}\cdot V_C\cdot J=\begin{pmatrix}
            I & -C\\
            0 & I\\                   
         \end{pmatrix}.
        \end{align*}
    As such, \cref{metadiagrecon} applies, $\mathcal{A}:(\Lambda\times 0)\mapsto (\Lambda\times 0)$, and so the two formulations of \cref{diagreconstr} coincide for the same lattice $\Lambda$.
\end{example}

The condition on the matrix $\mathcal{A}$ in \cref{metadiagrecon} is connected to the notion of Cohen's class distributions. Recall a distribution is in the Cohen's class if it is given by a convolution of theWigner distribution \cite{LuSk19}. Such distributions correspond to covariant $\mathcal{A}$-Wigner distributions \cite{CoRo22}. For the purpose of operator reconstruction, however, we can relax this condition to the more general $\mathcal{B}$-\textit{covariance property}:
\begin{definition}
    Let $\mathcal{B}\in GL(d,\mathbb{R})$. A symplectic matrix $\mathcal{A}$ gives rise to a $\mathcal{B}$-covariant quantisation if the property
    \begin{align*}
        \sigma^\mathcal{A}_{\alpha_z(S)} = T_{\mathcal{B}z} \sigma^\mathcal{A}_S
    \end{align*}
    holds for every $S\in\mathfrak{S}^\prime$.
\end{definition}
It is clear that in the case $\mathcal{B}=I$, $\mathcal{B}$-covariance is simply covariance as defined in \cite{CoRo22}. To see the connection between $\mathcal{B}$-covariant quantisations and operator reconstructions, we note the following:
\begin{proposition}
    A symplectic matrix $\mathcal{A}$ gives rise to a $\mathcal{B}$-covariant quantisation if and only if $\mathcal{A}$ has the block decomposition
    \begin{align}\label{bquantform}
        \begin{pmatrix}
            \mathcal{B} & A_2\\
            0 & \mathcal{B}^{-T}\\                   
         \end{pmatrix}.
    \end{align}
\end{proposition}
\begin{proof}
    Let $\mathcal{A}$ be of the form given by \cref{bquantform}. Writing the $\mathcal{B}$-covariance condition in terms of time-frequency shifts gives 
    \begin{align*}
        \sigma^\mathcal{A}_{\alpha_z(S)} = \rho(\mathcal{B}z,0) \sigma^\mathcal{A}_S.
    \end{align*}
    Using the definition of $\sigma^\mathcal{A}_S$ along with the intertwining property of metaplectic operators then gives
    \begin{align*}
        \sigma^\mathcal{A}_{\alpha_z(S)} &= \mu(\mathcal{A})\sigma_{\alpha_z (S)} \\
        &= \mu(\mathcal{A})\rho(z,0)\sigma_{S} \\
        &= \rho\big(\mathcal{A}(z,0)\big)\mu(\mathcal{A})\sigma_{S} \\
        &= \rho(\mathcal{B}z,0))\mu(\mathcal{A})\sigma_{S} \\
        &= T_{\mathcal{B}z} \sigma^\mathcal{A}_S.
    \end{align*}
    By the same argument, any $\mathcal{B}$-covariant quantisation must correspond to a symplectic matrix of the form \cref{bquantform}.
    
\end{proof}
By defining $\mathcal{A}$-quantisations in terms of the Weyl symbol, as opposed to the integral kernel, $\mathcal{B}$-covariant, and in particular covariant quantisations, have a simple geometric intuition. Furthermore, the form is clearly equivalent to the condition in \cref{metadiagrecon}, giving the following corollary:
\begin{corollary}
    Let $\mathcal{A}\in Sp(2d,\mathbb{R})$ give a $\mathcal{B}$-covariant quantisation, and let $g$, $\Lambda$, $Q$ and $S$ be as in \cref{diagreconstr}. Then given some operator $T\in\mathfrak{S}^\prime$ satisfying 
    \begin{align*}
        \mathrm{supp}(\sigma_T^{J\mathcal{A}}) \subset (1-\epsilon)Q,
    \end{align*}
    $T$ can be expressed as
    \begin{align*}
        T = \sum_{\lambda\in\Lambda^\prime} Q_{S^\prime} T(\lambda,\lambda) T_{\lambda} \alpha_\lambda (R)
    \end{align*}
    where $S^\prime = \mu(\mathcal{A})^*S$, $\mathcal{F}_W(R) = \mathcal{F}_{\Omega}\Big(\frac{g}{\mathcal{F}_W (\Check{S^\prime})}\Big)$ and $\Lambda^\prime\ = \mathcal{B}\Lambda$.
\end{corollary}

\section{From Reconstruction to Identification}
So far we have considered the case of reconstruction: Given a diagonal of the Gabor matrix, can one reconstruct the original operator. In this section, we relax this requirement to merely being able to distinguish operators based on the diagonal of the Gabor matrix. To this end, we look for sets of operators for which the mapping $T\mapsto \{Q_S T(\lambda,\lambda)\}_{\Lambda}$ is injective, for some lattice $\Lambda$. For rank--one operators, this is the problem of phase retrieval. For general operators, we call this property \emph{identifiability}:
\begin{definition}
    A set of operators $\mathcal{T}\subset\mathfrak{S}^\prime$ is \emph{identifiable} on the lattice $\Lambda$ with respect to window $S\in\mathfrak{S}$ if the map 
    \begin{align*}
        T\mapsto Q_S T(\lambda,\lambda)
    \end{align*}
    is injective from $\mathcal{T}$ to $\ell^\infty(\Lambda)$.
\end{definition}
Our terminology reflects that in this scenario, the diagonal of $Q_S T$ ``identifies'' $T$ in the set $\mathcal{T}$. We could, of course, take a different dual pairing than $(\mathfrak{S},\mathfrak{S}^\prime)$ to consider a broader class of windows, but for our purposes this definition suffices.

Since by the Weyl quantisation, the injectivity of the mapping $T\mapsto \{Q_S T(\lambda,\lambda)\}_{\Lambda}$ amounts to a question of completeness of translates of a function, the basis of results in this section will be Müntz-Szasz type theorems for higher dimensions. We recall the classical Müntz-Szasz theorem in one variable, in terms of exponential functions (the monomial case is equivalent):
\begin{theorem}[Müntz-Szasz Theorem]\label{muntzszasz1d}
    Let $\{\lambda_i\}_{i\in I}$ be a sequence of non-zero complex numbers, such that
    \begin{align}\label{muntzcond1d}
        \sum_{i\in I} \frac{1}{|\mathrm{Re}(\lambda_i)|} =\infty.
    \end{align}
    Then the set
    \begin{align*}
        \{e^{\lambda_i\cdot t}\}_{i\in I}
    \end{align*}
    spans a dense subset of the spaces $C([a,b])$ and $L^p([a,b])$ for $1\leq p < \infty$.
\end{theorem}
The complex analytic proof of \cref{muntzszasz1d} uses results on uniqueness sets for analytic functions with certain growth conditions, cf. Theorem 6.1, \cite{Lu71}. Extending to higher dimension is therefore a delicate matter, since the zero sets of analytic functions of several variables are generally less well--understood. A geometric approach in \cite{Kr94} uses the notion of a Müntz set in $\mathbb{R}^n$:
\begin{definition}
    Given some set $M\subset \mathbb{R}^d$, denote by $\Tilde{M}$ the affine hull of $M$, that is:
    \begin{align*}
        \Tilde{M} = \left\{ \sum_i a_i\cdot m_i\, : \, m_i\in M,\, a_i \in \mathbb{R},\, \sum_{i} a_i = 1 \right\}.
    \end{align*}
    The distance of $\Tilde{M}$ is given by dist$(\Tilde{M})= \inf_{m\in \Tilde{M}} \|m\|_{\mathbb{R}^d}$. We define \emph{Müntz sets} in the following inductive manner:
        \begin{enumerate}
            \item If $M$ is a point, then $M$ is a Müntz set.
            \item If dim$(M)=k > 0$, then $M$ is a Müntz set if $M$ contains countably many $(k-1)$ Müntz subsets $\{M_i\}_{i\in\mathbb{N}}$ such that $\Tilde{M_i}\neq \Tilde{M_j}$ if $i\neq j$, and the sum
            \begin{align*}
                \sum_{i\in\mathbb{N}} \frac{1}{1+\mathrm{dist}(\Tilde{M_i})}
            \end{align*}
            diverges.
        \end{enumerate}
\end{definition}
In the one--dimensional case, this is equivalent to \cref{muntzcond1d}, while in higher dimensions it enforces the spacing of points of a set in an analogous manner. The higher dimension Müntz-Szasz theorem is then as follows:
\begin{theorem}\label{multidimmuntz}
    Let $K\subset \mathbb{R}^n$ be a compact subset, and $M\subset \mathbb{R}^n$ an $n$--dimensional Müntz set. Then the set
    \begin{align*}
        \{e^{\lambda_m\cdot t}\}_{m\in M}
    \end{align*}
    spans a dense subset of the spaces $C(K)$ and $L^p(K)$ for $1\leq p < \infty$.
\end{theorem}
We will use this result insofar as it relates to lattices:
\begin{corollary}\label{latticemuntzmulti}
    Given any lattice $\Lambda\subset\mathbb{R}^{n}$, the set
    \begin{align*}
        \{e^{\lambda\cdot t}\}_{\lambda\in \Lambda}
    \end{align*}
    spans a dense subset of the spaces $C(K)$ and $L^p(K)$ for $1\leq p < \infty$.
\end{corollary}
\begin{proof}
    This follows from \cref{multidimmuntz} upon verifying that any lattice $\Lambda$ is a Müntz set.
    
\end{proof}
In the sequel, we denote by $t\cdot A x$ the dot product $t^T A x$ as opposed to the inner product, so for a complex $A=B+iC$, $t\cdot Ax = t\cdot Bx + it\cdot Cx = x\cdot B^T t + ix\cdot C^T t = x\cdot A^T t$. In particular, symmetric matrices, possibly complex valued, have the property $t\cdot Ax = x\cdot At$, as opposed to Hermitian matrices in the inner product setting.

The equivalent of Zalik's theorem \cite{Za78} follows for higher dimensions:
\begin{theorem}[Multivariable Zalik's Theorem, Theorem 6.1 \cite{GrLiSh22}]\label{multivariable-zalik}
    Let $A\in \mathbb{C}^{n\times n}$ be a Hermitian matrix with Re$(A)$ non-singular. Then for any lattice $\Lambda\subset\mathbb{R}^n$, the set
    \begin{align*}
        \{e^{-(x-\lambda)A(x-\lambda)}\}_{\lambda\in\Lambda}
    \end{align*}
    spans a dense subset of $C(K)$ and $L^p(K)$ for $1\leq p < \infty$. 
\end{theorem}
\begin{proof}
    For each $\lambda\in\Lambda$, $e^{-\lambda A \lambda}$ is a non-zero constant, and $\{e^{-xAx + 2x\mathrm{Re(A)\lambda }}\}_{\lambda\in\Lambda}$ is dense in $C(K)$ and $L^p(K)$ if and only if $\{e^{ 2x\mathrm{Re(A)\lambda }}\}_{\lambda\in\Lambda}$ is. The result then follows from \cref{latticemuntzmulti}, since Re$(A)$ being non-singular ensures $Re(A)\Lambda$ is again a lattice in $\mathbb{R}^n$.
    
\end{proof}

Using \cref{latticemuntzmulti}, we can take a similar approach to \cref{metadiagrecon} to find classes of identifiable operators: 
\begin{theorem}\label{metaplectic-identification-i}
    Let $\mathcal{A}\in Sp(2d,\mathbb{R})$ have the block decomposition
    \begin{align}
        \begin{pmatrix}
            A_1 && A_2 \\
            0 && A_1^{-T}
        \end{pmatrix}
    \end{align}
    where $A_i$ are $2d\times 2d$ matrices. Then the space 
    \begin{align*}
        \mathcal{P}_{\mathcal{A},K} := \{S\in\mathcal{HS} : \mathrm{supp}\big(\sigma^{\mathcal{A}}_S\big) \subset K\}
    \end{align*}
    is identifiable on \emph{any} lattice $\Lambda\subset\mathbb{R}^{2d}$, with respect to the window $S=\mu(\mathcal{A})^*\left(\varphi_A\otimes \varphi_A\right)$, where $\varphi_A$ is the generalised Gaussian  
    \begin{align*}
        \varphi_A(t) := 2^{d/4}e^{-\pi t At}
    \end{align*}
    for some symmetric $A\in GL(d,\mathbb{R})$.
\end{theorem}

\begin{proof}
    We begin by showing that the set 
    \begin{align*}
        \mathcal{P}_{I,K} := \{T\in\mathcal{HS} : \mathrm{supp}\big(\sigma_T\big) \subset K\}
    \end{align*}
    is identifiable on any lattice $\Lambda\subset \mathbb{R}^{2d}$. We again use that 
    \begin{align*}
        Q_S T(\lambda,\lambda) = \langle \sigma_T, T_\lambda \sigma_S\rangle_{\mathcal{HS}}.
    \end{align*}
    Then for symmetric non-singular $A$, the operator $S=\varphi_A\otimes \varphi_A$ has the Weyl symbol 
    \begin{align*}
        \sigma_S = 2^{d/4}e^{-\pi t Bt},
    \end{align*}
    where $B$ is a symmetric non-singular matrix, given explicitly by Proposition 244, \cite{DeGo11}. It follows from \cref{multivariable-zalik} that the set $\{T_\lambda \sigma_S\}_{\lambda\in\Lambda}$ is dense in $L^2(K)$ for any $\Lambda\subset \mathbb{R}^{2d}$, and so the set $\{\alpha_\lambda (S)\}_{\lambda\in\Lambda}$ is dense in the set $\mathcal{P}_{I,K}$, and hence $\mathcal{P}_{I,K}$ is identifiable on any $\Lambda$ with respect to $S$. The general $\mathcal{A}$ case follows using the same argument as in \cref{metadiagrecon}, and noting that the transformation $\mathcal{A}\Lambda$ gives a new lattice, and since \cref{multivariable-zalik} requires only that $\mathcal{A}\Lambda$ is a lattice, the result follows.
    
\end{proof}

For complex matrices, we can use the following extension of Zalik's theorem to higher dimension:
\begin{lemma}\label{diagonalcomplexdensity-nonsingular}
    Let $A\in \mathbb{C}^{n\times n}$ be diagonalisable with respect to a real orthogonal matrix, such that Re$(A)$ is non-singular, and let $K$ be a compact subset of $\mathbb{R}^n$. Then there exists a real lattice $\Lambda\subset \mathbb{R}^n$ such that the set 
    \begin{align*}
        \{e^{tA\lambda}\}_{\lambda\in\Lambda}
    \end{align*}
    spans a dense subset of $C(K)$ and $L^p(K)$ for $1\leq p < \infty$.
\end{lemma}
\begin{proof}
    Let $A=UD_A U^T$, where $U$ is an orthogonal real matrix and $D_A$ is a complex diagonal matrix. Set $\Lambda = U\mathbb{Z}^n$. Then  
    \begin{align*}
        \{e^{tA\lambda}\}_{\lambda\in\Lambda} = \{e^{tU D_A x}\}_{x\in\mathbb{Z}^n}.
    \end{align*}
    Since $K$ is compact for any change of variables, we make the change of variables $U^T t\mapsto t$, so the problem reduces to the density of the span of 
    \begin{align}\label{lemmadensityform}
        \{e^{t D_A x}\}_{x\in\mathbb{Z}^n}.
    \end{align}
    Since $D_A$ is diagonal, the density of \cref{lemmadensityform} is satisfied if the restriction to each coordinate spans a dense set in $C(I)$, where $U^T K\subset I^n$. This follows from the $1$-dimensional Müntz-Szasz \cref{muntzszasz1d}, since by our assumption that Re$(A)$ is non-singular, the condition \cref{muntzcond1d} holds.
    
\end{proof}

In the case of a singular real part of the matrix $A$, we have a similar result, but we then require a density condition on the lattice in addition:
\begin{lemma}\label{diagonalcomplexdensity-singular}
     Let $A\in \mathbb{C}^{n\times n}$ be non--singular and diagonalisable with respect to a real orthogonal matrix, and let $K$ be a compact subset of $\mathbb{R}^n$. Then there exists a real lattice $\Lambda\subset \mathbb{R}^n$ such that the set 
    \begin{align*}
        \{e^{tA\lambda}\}_{\lambda\in\Lambda}
    \end{align*}
    spans a dense subset of $L^2(K)$.
\end{lemma}
\begin{proof}
    Again consider the singular value decomposition $A=UD_A U^T$, for $U$ an orthogonal real matrix and $D_A$ a complex diagonal matrix. We can assume without loss of generality that the first $k$ entries of the diagonal $D_A$ have non-zero real part, while the remaining $n-k$ are purely imaginary, where $0\leq k\leq n$. Let $I\subset \mathbb{R}^k$ and $J\subset \mathbb{R}^{n-k}$ such that $K\subset I\times J$.
    
    We then take $\Lambda=\Lambda_1\times \Lambda_2$, where $\Lambda_1=U\mathbb{Z}^k$, and $\Lambda_2 = cU\mathbb{Z}^{n-k}$, where $c$ is a positive constant such that $U^T J$ is contained in a fundamental domain of $\Lambda_2^\perp$ translated by a lattice point. Then by \cref{diagonalcomplexdensity-nonsingular}, the set
    \begin{align*}
        \{e^{t \lambda}\}_{\lambda\in\Lambda_1}
    \end{align*}
    spans a dense subset of $L^2(I)$, while by the well known result of Duffin and Schaeffer \cite{Du52}, 
    \begin{align*}
        \{e^{t \lambda}\}_{\lambda\in\Lambda_2}
    \end{align*}
    spans a dense subset of $L^2(J)$. Hence, the product of the two spans a dense subset of $L^2(K)$.
    
\end{proof}

We will need the following lemma for the next result of this section:
\begin{lemma}[Lemma 1, \cite{JiTaZh20}]\label{metadecomp}
    Any symplectic matrix $\mathcal{C}\in Sp(2d,\mathbb{R})$ has a decomposition 
    \begin{align*}
    \mathcal{C} = D_L V_A V_B^T V_C.
\end{align*}
\end{lemma}
\begin{proof}
    Recall that for any free symplectic matrix $\mathcal{B}$, we can decompose $\mathcal{B}$ as
\begin{align*}
    \mathcal{B} = V_A D_L J V_B
\end{align*}
where $A,B$ are symmetric matrices and $L$ non-singular (Proposition 2.39, \cite{Go06}). Furthermore, since any symplectic matrix can be expressed as the composition of two free symplectic matrices (Proposition 2.36, \cite{Go06}), we find, after a simple calculation (and relabelling the matrices), that for a general $\mathcal{C}\in Sp(2d,\mathbb{R})$;
\begin{align*}
    \mathcal{C} = D_L V_{\Tilde{A}} V_{\Tilde{B}}^T V_A,
\end{align*}
as required.

\end{proof}

This decomposition is useful in particular because it has an explicit algorithm to calculate (Algorithm 1, \cite{JiTaZh20}). As such, the condition in the next theorem is easily verified.

\begin{theorem}\label{metaplectic-identification-ii}
    Let $\mathcal{A}\in Sp(2d,\mathbb{R})$, such that $A,\Tilde{B}$ corresponding to $\mathcal{A}$ in the decomposition in \cref{metadecomp} are simultaneously diagonalisable. Then for any compact $K\subset\mathbb{R}^{2d}$, there exists some lattice $\Lambda = A\mathbb{Z}^{2d}$ such that $\mathcal{P}_{\mathcal{A},K}$ is identifiable on $\Lambda$ with respect to the Gaussian $S_0=\varphi_0\otimes\varphi_0$.
\end{theorem}
\begin{proof}
Since the metaplectic operators are unitary, 
\begin{align*}
    \langle T, \alpha_x (S)\rangle &= \langle \sigma_T, T_x \sigma_S \rangle \\
    &= \langle \mu(\mathcal{A})\sigma_T, \mu(\mathcal{A})T_x \sigma_S \rangle.
\end{align*}
We decompose $\mathcal{A}$ as in \cref{metadecomp}:
\begin{align*}
    \mathcal{A} = D_L V_{\Tilde{A}} V_{\Tilde{B}}^T V_A.
\end{align*}
Inserting this into the inner product gives
\begin{align*}
    \langle T, \alpha_x (S) \rangle = \langle \mu(D_L V_{\Tilde{A}})^* \mu(\mathcal{A})\sigma_T, \mu(V_{\Tilde{B}}^T V_A)T_x \sigma_S \rangle.
\end{align*}
Since $\mu(D_L V_{\Tilde{A}})^*$ maps compactly supported functions on $K$ unitarily to compactly supported functions on $L^{-1}\cdot K$, if $\mu(V_{\Tilde{B}}^T V_A)T_x \sigma_S$ is dense in $L^2(L^{-1}\cdot K)$ then we have operator identification. Since $K$ is an arbitrary compact subset, we may consider $K$ instead of $L^{-1}\cdot K$ here for simplicity of notation.

We consider the window $S$ is given by 
\begin{align*}
    \sigma_S(t) = e^{-\pi t\cdot t}
\end{align*}
for ease of notation. The product $e^{i\pi t\cdot At}\cdot e^{-\pi(t-x)\cdot (t-x)}$ is then
\begin{align*}
    e^{it\cdot At}\cdot e^{(t-x)\cdot (t-x)} = e^{-\pi\left(t\cdot (I-iA)t - 2t\cdot x + x\cdot x\right)}.
\end{align*}
If $\Tilde{B}$ is singular, then the action of $\mu(V_{\Tilde{B}}^T)$ is given by convolution over the space of variables $\Tilde{B}\mathbb{R}^{2d}$, while the variables in the space ker$(\Tilde{B})$ are unchanged. We denote $B=\Tilde{B}^{-1}$, considered to be acting on $\mathbb{R}^{2d} / \mathrm{ker}(\Tilde{B})$. Then for  $f\in L^2(\mathbb{R}^{2d}) = L^2(\mathrm{ker}(\Tilde{B})) \otimes L^2(\Tilde{B}\mathbb{R}^{2d})$;
\begin{align*}
    \mu(V_{\Tilde{B}}^T) f = |B|\cdot (\delta \otimes e^{i\pi Bt\cdot t}) * f.
\end{align*}

We can assume without loss of generality that $K$ is a product of the form $X\times Y$ where $X\subset B\mathbb{R}^{2d}$ and $Y \subset \mathrm{ker}(B)$, by simply enlarging $K$ if necessary. We then decompose $L^2(K)$ as $L^2(X) \otimes L^2(Y)$.
The span of the functions 
\begin{align*}
    e^{-\pi\left(t\cdot (I-iA)t - 2t\cdot x + x\cdot x\right)}
\end{align*}
is dense in $L^2(Y)$ for $x$ in any $\Lambda\subset \mathrm{ker}(B)$. This is because multiplying a function in $L^2(Y)$ by $e^{-\pi t\cdot(I-iA)t}$ is an injective map on $L^2(Y)$, and $e^{-\pi x\cdot x}$ is a non-zero constant, and then the span of $e^{2\pi t\cdot x}$ is dense in $L^2(Y)$ by \cref{multivariable-zalik}. 

We then consider the case that $B$ is non-singular, and see that
\begin{align*}
    e^{i\pi t\cdot Bt} * (e^{i\pi t\cdot At}\cdot e^{-\pi(t-x)^2}) &= e^{- \pi x\cdot x} \int_{\mathbb{R}^{2d}} e^{i\pi t'\cdot Bt'}\cdot e^{-\pi\left((t-t')\cdot (I-iA)(t-t') - 2(t-t')\cdot x \right)}\, dt' \\
    &= e^{- \pi x\cdot x} \int_{\mathbb{R}^{2d}} e^{-\pi\left(t'\cdot(I-iA-iB) t' - 2t\cdot(I-iA) t'\right)}\\
    &\qquad\times  e^{-\pi\left( t\cdot (I-iA)t + 2t'\cdot x - 2t\cdot x\right)}\, dt' \\
    &= e^{- \pi\left(t\cdot (I-iA)t -2t\cdot x + x\cdot x\right)} \\
    &\qquad\times \int_{\mathbb{R}^{2d}} e^{-\pi\left(t'\cdot(I-iA-iB) t' - 2t\cdot(I-iA) t' + 2t'\cdot x \right)}\, dt'.
\end{align*}
Denoting $D=(1-iA-iB)$, which is symmetric and invertible, and completing the square gives
\begin{align*}
    \int_{\mathbb{R}^{2d}} e^{-\pi\left(t'\cdot D t' - 2t\cdot(I-iA) t' + 2t'\cdot x \right)}\, dt' &= e^{\pi [(I-iA)t+x]\cdot D^{-1}[(I-iA)t+x]} \\
    &\quad\times \int_{\mathbb{R}^{2d}} e^{-\pi\left((t'-D^{-1}[(I-iA)t+x])\cdot D (t'-D^{-1}[(I-iA)t+x])  \right)}\, dt'.
\end{align*}
Since $D$ is symmetric, the multivariate complex Gaussian integral can be computed to be: 
\begin{align*}
     \int_{\mathbb{R}^{2d}} e^{-\pi\left(t'\cdot D t' - 2t\cdot(I-iA) t' + 2t'\cdot x \right)}\, dt'= \frac{1}{|D|^{1/2}}.
\end{align*}
Combining the two calculations then gives
\begin{align}\label{density-form}
    e^{i\pi t\cdot Bt} * (e^{i\pi t\cdot At}\cdot e^{-\pi(t-x)^2}) &= \frac{1}{|D|^{1/2}} e^{-\pi\left(t\cdot (I-iA)t -2t\cdot x + x\cdot x + [(I-iA)t+x]\cdot D^{-1}[(I-iA)t+x] \right)} \nonumber \\
    &= \frac{1}{|D|^{1/2}} e^{-\pi t\cdot (I-iA)[I + D^{-1}(I-iA)]t}\cdot \nonumber\\
    &\qquad \times e^{2\pi t\cdot[I - (I-iA)D^{-1}]x}e^{-\pi x\cdot(I-D^{-1})x}.
\end{align}

Using the same argument as above, the first term of \cref{density-form} is a non-zero bounded function, independent of $x$, for any compact subset $K$, so the composition with any compactly supported function $f$ is again a compactly supported function, as the composition mapping is injective. The third term is a non-zero constant for each $x$. Hence, density of the span of the set 
\begin{align*}
    \{e^{i\pi t\cdot Bt} * (e^{i\pi t\cdot At}\cdot e^{-\pi(t-x)^2})\}_{x\in\Lambda}
\end{align*}
in the space of $L^2(K)$ is equivalent to the span of the density of the middle term of \cref{density-form}.
Note that
\begin{align*}
    (I-iA)D^{-1} &= (I-iA - iB + iB)(I-i(A+B))^{-1} \\
    &= I + iB(I-i(A+B))^{-1} \\
    &= I + iBD^{-1},
\end{align*}
so the middle term of \cref{density-form} reduces to $e^{-2\pi t\cdot iBD^{-1}x}$. Since we assume $B$ is invertible, $BD^{-1}$ is invertible. By assumption, $A$ and $B$ are simultaneously diagonalisable, that is to say $A=UD_A U^T$ and $B=UD_BU^T$ for the same orthogonal matrix $U$. Then since $I=UIU^T$ for any orthogonal $U$, $D$ and $D^{-1}$ are also diagonalisable with respect to $U$. Hence, the result follows from \cref{diagonalcomplexdensity-singular}.

\end{proof}

As a most basic consequence of \cref{metaplectic-identification-ii}, the case $\mathcal{A}=I$ satisfies the condition, so we recover the same result as for \cref{metaplectic-identification-i}:
\begin{corollary}
    The space $\mathcal{P}_{I,K}$, of those operators with compactly supported Weyl symbols, is identifiable for any lattice $\Lambda \subset \mathbb{R}^{2d}$.
\end{corollary}
Similarly, in the case $\mathcal{A} = J$, we recover the well known fact of operator identification for band-limited symbols, discussed in previous sections.

We can also consider the problem of phase retrieval for compactly supported functions, as discussed in \cite{GrLi23} \cite{GrLiSh22}:
\begin{example}\normalfont
    Consider the case of phase retrieval on a lattice for a compactly support $f \in L^2(K)$. This is equivalent to the identification of the operator $f\otimes f$ with respect to a rank--one window $g\otimes g$. The rank--one operator $f\otimes f$ being compactly supported corresponds to $\mathcal{F}_{\mathcal{U}}(f\otimes f)$ being compactly supported, where $\mathcal{U}$ corresponds to the metaplectic operator mapping the Weyl symbol to kernel;
    \begin{align*}
        \mathcal{U} = D_L \cdot \mathcal{A}_{FT_2},
    \end{align*}
    where 
    \begin{align*}
        L=\begin{pmatrix}
        I/2 && I/2 \\
        I && -I
    \end{pmatrix}
    \end{align*}
    and
    \begin{align*}
        \mathcal{A}_{FT_2} = \begin{pmatrix}
        I && 0 && 0 && 0 \\
        0 && 0 && 0 && -I\\
        0 && 0 && I && 0 \\
        0 && I && 0 && 0
    \end{pmatrix}.
    \end{align*}
    $\mathcal{A}_{FT_2}$ can be decomposed in the form of \cref{metadecomp} as
    \begin{align*}
        \mathcal{A}_{FT_2} = 
        \begin{pmatrix}
            C && 0  \\
            0 && C^{-1} 
        \end{pmatrix}
        \begin{pmatrix}
            I && 0  \\
            A && I 
        \end{pmatrix}
        \begin{pmatrix}
            I && B  \\
            0 && I
        \end{pmatrix}
        \begin{pmatrix}
            I && 0  \\
            A && I
        \end{pmatrix},
    \end{align*}
    with $A=\begin{pmatrix}
            0 && 0  \\
            0 && -I
        \end{pmatrix}$,
    $B=\begin{pmatrix}
            0 && 0  \\
            0 && I
        \end{pmatrix}$
    and $C=\begin{pmatrix}
            I && 0  \\
            0 && -I
        \end{pmatrix}$.
    From \cref{metaplectic-identification-ii}, it follows that the space $\mathcal{P}_{\mathcal{A},K}$ is identifiable on the lattice $\Lambda_1 \times \Lambda_2$, where $\Lambda_1$ can be taken to be any lattice in $\mathbb{R}^d$, while $\Lambda_2$ is such that $Y \subset \Lambda^\perp$ for $C\cdot K \subset X\times Y$, $X,Y\subset \mathbb{R}^d$.

    In particular since the window $e^{-\pi t^2}$ is rank--one, this gives phase retrieval for the space $L^2(K)$ on the lattice $\Lambda_1 \times \Lambda_2$. Furthermore, consider the case where $S=f\otimes f$, and $\mu(A)f$ is compactly supported. The metaplectic operator $\mu(A)$, with $A\in Sp(d,\mathbb{R})$, acting on $f$, corresponds to the metaplectic operator $\mu(\mathcal{A})$ acting on $K_{f\otimes f}$, where 
    \begin{align*}
        \mathcal{A} = \begin{pmatrix}
            A && 0 \\
            0 && A^{-T}
        \end{pmatrix}.
    \end{align*}
    Thus, it can easily be seen that the ability to identify, and consequently perform phase retrieval, applies to any of the sets of operators $S= f\otimes f$ where $\mu(A)f$ is compactly supported for any $A\in Sp(d,\mathbb{R})$.
\end{example}

\begin{remark}\normalfont
    The condition of simultaneously diagonalisable $A,\Tilde{B}$ in \cref{metaplectic-identification-ii} is required in order to make the real change of variables in \cref{diagonalcomplexdensity-singular}. The complex analytical proof of the Müntz-Szasz theorem relies on sets of uniqueness for complex analytic functions with a certain growth condition (cf. \cite{Lu71}). A general change of variables in \cref{diagonalcomplexdensity-singular} would not necessarily preserve this growth condition in the corresponding analytic function. It is possible that a similar result for sets of uniqueness for complex analytic functions of general exponential type would suffice to relax the diagonalisation condition in \cref{metaplectic-identification-ii}.
\end{remark}

\section*{Acknowledgements}
The author would like to thank Franz Luef for the helpful discussions.

\bibliography{refs}
\bibliographystyle{abbrv}

\end{document}